\documentclass[oneside,english]{amsart}
 \usepackage{geometry}
\usepackage{babel}

\usepackage{amsthm}
\usepackage{amstext}
\usepackage{amssymb}
\usepackage[unicode=true]
 {hyperref}
\usepackage{breakurl}
\usepackage{url}

\makeatletter

\providecommand{\tabularnewline}{\\}

\numberwithin{equation}{section}
\numberwithin{figure}{section}
\theoremstyle{plain}
\newtheorem{thm}{Theorem}
  \theoremstyle{definition}
  \newtheorem{defn}[thm]{Definition}
  \theoremstyle{plain}
  \newtheorem*{thm*}{Theorem}
  \theoremstyle{plain}
  \newtheorem{conjecture}[thm]{Conjecture}
  \theoremstyle{remark}
  \newtheorem*{rem*}{Remark}
  \theoremstyle{definition}
  \newtheorem*{example*}{Example}
  \theoremstyle{plain}
  \newtheorem{prop}[thm]{Proposition}
  \theoremstyle{plain}
  \newtheorem{lem}[thm]{Lemma}
  \theoremstyle{remark}
  
  \theoremstyle{plain}
  \newtheorem{cor}[thm]{Corollary}

\usepackage{amsmath}
\allowdisplaybreaks
\usepackage{longtable}

\newenvironment{thmcust}[2][Theorem]{\begin{trivlist}
\item[\hskip \labelsep {\bfseries #1}\hskip \labelsep {\bfseries #2.}]}{\end{trivlist}}

\AtBeginDocument{
  
}
\def\blfootnote{\xdef\@thefnmark{}\@footnotetext}

\makeatother

\begin{document}

\title{Counting $G$-Extensions by Discriminant}

\author{Evan P. Dummit}
\address{Dept. of Mathematics, University of Rochester,  915 Hylan Building, RC Box 270138, Rochester, NY 14627}
\email{edummit@ur.rochester.edu}

\begin{abstract}
The problem of analyzing the number of number field extensions $L/K$
with bounded (relative) discriminant has been the subject of renewed
interest in recent years, with significant advances made by Schmidt,
Ellenberg-Venkatesh, Bhargava, Bhargava-Shankar-Wang, and others. In this
paper, we use the geometry of numbers and invariant theory of finite
groups, in a manner similar to Ellenberg and Venkatesh, to give an
upper bound on the number of extensions $L/K$ with fixed degree,
bounded relative discriminant, and specified Galois closure.
\end{abstract}
\maketitle
\section{Overview}

Over a century ago, Hermite showed that the number of number fields
of a given degree whose (absolute) discriminant is less than $X$
is finite. Thus, ordering number fields of a fixed degree (or fixed
Galois closure) by discriminant provides us with a variety of well-posed
counting problems.

For a fixed number field $K$, our primary interest is in analyzing
the asymptotics, as $X\to\infty$, of the number of extensions $L/K$,
of fixed degree and Galois closure, whose discriminant (norm) is less
than $X$. Providing exact asymptotics is quite difficult and has
been carried out in only a few cases.

In Section \ref{sec:Notation-and-Background}, we briefly review a
number of results on counting number fields by discriminant, and then in Section \ref{sec:Appendix-Invariants}
we review some necessary background on polynomial invariants attached
to representations of finite groups. 

In Section \ref{sec:Proof-of-permreps}, we then prove a general theorem
bounding from above the number of extensions of a given degree, bounded
discriminant, and specified Galois closure.  We give a prototypical example in Section \ref{sec:A-Generic-Example}, and then finish with some concluding remarks.

\section{Notation and Background\label{sec:Notation-and-Background}}

To introduce some notation, let $K$ be a number field and $L/K$
be an extension of degree $n$. Also let $\mathcal{O}_{L}$ and $\mathcal{O}_{K}$
be the rings of integers, and $D_{L}$ and $D_{K}$ be the absolute
discriminants, of $L$ and $K$ respectively. We also take $\mathrm{Nm}_{K/\mathbb{Q}}$
to be the absolute norm on ideals and $\mathcal{D}_{L/K}$ to be the
relative discriminant ideal.

We will understand $f(X)\sim g(X)$ to mean that ${\displaystyle \lim_{X\to\infty}{\displaystyle \dfrac{g(X)}{f(X)}}=1}$,
and $f(X)\ll g(X)$ to mean that $f(x)<c\, g(X)$ for some constant
$c>0$ and $X$ sufficiently large (where $c$ may depend on other
parameters such as $n$ and $\epsilon$ that will be clear from the
context). The group $G$ will also always refer to a finite group
equipped with an embedding into $S_{n}$, and is to be interpreted
as the Galois group of the Galois closure of $L/K$.

\subsection{Counting Extensions of Fixed Degree}
Our central problem is to count extensions $L/K$ where $[L:K]=n$.
\begin{defn}
For a fixed $K$ and $n$, we define $N_{K,n}(X)$ to be the number
of number fields $L$ (up to $K$-isomorphism) with extension degree
$[L:K]=n$ and absolute discriminant norm $\mathrm{Nm}_{K/\mathbb{Q}}(\mathcal{D}_{L/K})<X$.
\end{defn}

A folk conjecture, sometimes attributed to Linnik, says that 
\begin{equation*}
N_{K,n}(X)\sim C_{K,n}X\label{eq:Linnik}
\end{equation*}
for fixed $n$ and as $X\to\infty$, for some positive constant $C_{K,n}$
depending on $K$ and $n$. Even for the base field $K=\mathbb{Q}$,
the best known results for large $n$ are far away from this conjectured
result. Only in some low-degree cases $(n\leq5$) is this conjecture
proven: for general $K$, the case $n=2$ is an exercise in Kummer
theory, and the case $n=3$ for $K=\mathbb{Q}$ is due to Davenport
and Heilbronn \cite{Davenport-Heilbronn}, while for general $K$
it is due to Datskovsky and Wright \cite{Datskovsky-Wright}. For
$K=\mathbb{Q}$, the results for $n=4$ and $n=5$ are also known
and due to Cohen-Diaz y Diaz-Olivier and Bhargava \cite{Bhargava-3,Bhargava-4, Cohen-1}, (a slightly weaker exponent was first established by Kable-Yukie \cite{Yukie,Kable-Yukie}), and for general $K$ they are due to Bhargava-Shankar-Wang \cite{Bhargava-Shankar-Wang}.
However, these techniques are not expected to extend to higher-degree
extensions.

Our starting point for counting extensions of higher degree is the
following theorem of Schmidt \cite{Schmidt}:
\begin{thm*}[Schmidt]
For all $n$ and all base fields $K$,\begin{equation}
N_{K,n}(X)\ll X^{(n+2)/4}.\label{eq:Schmidt}\end{equation}

\end{thm*}
The approach of Schmidt can be broadly interpreted as follows: if
$L/K$ is an extension of degree $n$, first use Minkowski's Lattice
Theorems to obtain an element $\alpha\in\mathcal{O}_{L}$ whose archimedean
norms are small (in terms of $X$). This gives bounds on the coefficients
of the minimal polynomial of $\alpha$; counting the number of possibilities
for $\alpha$ yields the upper bound on the number of possible extensions
$L/K$. We will note that some care is necessary in the above argument:
in fact, Schmidt actually counts chains of primitive extensions $K\subset L_{1}\subset\cdots\subset L_{t-1}\subset L$ to avoid possible issues arising from the existence of large-degree subfields.  (The overall exponent in $X$, ultimately, is independent of any assumption of primitivity.)

The best upper bound for general $n$ was established by Ellenberg
and Venkatesh \cite{EV-numberfields}:
\begin{thm*}[Ellenberg-Venkatesh]
\label{thm:(Ellenberg-Venkatesh)}  For all $n>2$
and all base fields $K$, \[
N_{K,n}(X)\ll(X\, D_{K}^{n}\, A_{n}^{[K:\mathbb{Q}]})^{\exp(C\sqrt{\log\, n})},\]
where $A_{n}$ is a constant depending only on $n$ and $C$ is an
absolute constant.
\end{thm*}
Although the constants are not explicitly computed in the paper, after
some effort one can show that for sufficiently large $n$ (roughly
on the order of $n=20$), the result becomes stronger than Schmidt's
bound.

By taking logarithms, one may recast Theorem \ref{thm:(Ellenberg-Venkatesh)}
as showing that \[
\limsup_{X\to\infty}\dfrac{\log\, N_{K,n}(X)}{\log\, X}\ll n^{\epsilon}\]
for any $\epsilon>0$. For comparison, Schmidt's result is that this limit is at most $\dfrac{n+2}{4}$, while Linnik's conjecture is that this limit is 1.

Ellenberg-Venkatesh use a modification of Schmidt's technique: rather
than counting the number of possibilities for a single element of
$\mathcal{O}_{L}$, they instead count linearly-independent $r$-tuples
of elements of $\mathcal{O}_{L}$, where $r$ is chosen at the end
so as to optimize the resulting bound. Then by using properties of
the invariant theory of products of symmetric groups, they rephrase
the problem into one about counting integral points on a scheme which
is a generically-finite cover of affine space.

\subsection{Counting Extensions with Specified Galois Closure}

We may refine the basic counting problem by restricting our attention
to extensions $L/K$ whose Galois closure $\hat{L}/K$ has Galois
group isomorphic to a particular finite permutation group $G$.
\begin{defn}
For fixed $K$ and $n$, and a transitive permutation group $G\hookrightarrow S_{n}$
with a given embedding into $S_{n}$, we define $N_{K,n}(X;G)$ to
be the number of number fields $L$ (up to $K$-isomorphism) such
that \end{defn}
\begin{enumerate}
\item The degree $[L:K]=n$,
\item The absolute norm of the relative discriminant $\mathrm{Nm}_{K/\mathbb{Q}}(\mathcal{D}_{L/K})$
is less than $X$, and 
\item The action of the Galois group of the Galois closure of $L/K$ on
the complex embeddings of $L$ is permutation-isomorphic to $G$.
\end{enumerate}
Extensions satisfying these conditions are referred to
as $G$-extensions. It is also common to abuse terminology and
refer to $G$ as the ``Galois group'' of the extension $L/K$,
despite the fact that this extension is not typically Galois.

A series of conjectures of Malle \cite{Malle-1,Malle-2} give expected
growth rates for $N_{K,n}(X;G)$ depending on the group $G$. Explicitly,
for $G$ a transitive subgroup acting on $\Omega=\left\{ 1,2,\dots,n\right\} $,
and for $g$ in $G$, define the index of an element \[
\mathrm{ind}(g)=n-\left[\text{number of orbits of }g\text{ on }\Omega\right],\]
which is also equal to the sum, over all cycles in the cycle decomposition of $g$ in $S_{n}$, of the length of the cycle minus 1.
Next define the index of $G$ to be \[
\mathrm{ind}(G)=\min\left\{ \mathrm{ind}(g)\,:\,1\neq g\in G\right\} .\]
We also set \[
a(G)=1/\mathrm{ind}(G).\]
Note that the index of a transposition is equal to 1, and (since an
element with index 1 has $n-1$ orbits) the transpositions are the
only elements of index 1.

The absolute Galois group of $K$ acts on the conjugacy classes of
$G$ via the action on $\bar{\mathbb{Q}}$-characters of $G$. We
define the orbits (of that action) to be the {}``$K$-conjugacy classes''
of $G$. Since all elements in a $K$-conjugacy class have the same
index, we define the index of a conjugacy class to be the index of
any element in that class.

The strong form of Malle's conjecture is as follows:
\begin{conjecture}
(Malle, strong form) \label{con:(Malle,-strong-form)}There exists
a constant $c(k,G)>0$ such that \[
N_{K,n}(X;G)\sim c(K,G)\cdot X^{a(G)}\cdot\log(X)^{b(K,G)-1},\]
where $a(G)=\dfrac{1}{\mathrm{ind}(G)}$ and $b(K,G)=\#\left\{ C:\, C\text{ a }K\text{-conjugacy class of minimal index ind}(G)\right\} $. \end{conjecture}
\begin{rem*}
We would expect by Linnik's conjecture that for any
group $G$, the asymptotics should not exceed $X^{1}$, and indeed
it is not hard to see (cf. Lemma 2.2 of \cite{Malle-2}) that if $a(G)=1$
then $b(K,G)$ is also 1.
\end{rem*}
The strong form of Malle's conjecture holds for all abelian groups;
this is a result of Wright \cite{Wright}. However, Kl{\"u}ners \cite{Kluners}
has constructed a counterexample to the $\log(X)$ part of the conjecture
for the nonabelian group $G=C_{3}\wr C_{2}$ of order 18 embedded
in $S_{6}$. (Kl{\"u}ners also notes that this is not a unique example,
and that all groups of the form $C_{p}\wr C_{2}$ yield counterexamples
to Malle's conjecture as formulated above.) The ultimate difficulty
is the potential existence of an intermediate cyclotomic subfield
inside the extension: in this case, $\mathbb{Q}(\zeta_{3})$ (or $\mathbb{Q}(\zeta_{p})$
in the general family).

There is a recent refinement of the exponent of the log-term in Malle's
conjecture over function fields, due to T{\"u}rkelli \cite{Turkelli},
which appears to avoid all of the known counterexamples. T{\"u}rkelli's
refinement is motivated by counting points on components of non-connected
Hurwitz schemes. The question of counting points on connected Hurwitz
schemes was related to counting extensions of function fields in a
paper of Ellenberg-Venkatesh \cite{EV-functionfields}, and their
heuristics (subject to some assumptions) aligned with Malle's. T{\"u}rkelli
extended their arguments to cover non-connected Hurwitz schemes, and
the difference in the results compared to those of Ellenberg-Venkatesh
suggested a modification to Malle's conjecture.

It is generally believed that the power of $X$ in Malle's conjecture
is essentially correct. Explicitly: 
\begin{conjecture}
(Malle, weak form) \label{con:(Malle,weak)} For any $\epsilon>0$
and any number field $K$, $X^{a(G)}\ll N_{K,n}(X;G)\ll X^{a(G)+\epsilon}$,
where $a(G)=\dfrac{1}{\mathrm{ind}(G)}$.
\end{conjecture}
If true, Malle's conjecture, even when we restrict to the {}``weak
form'' that only considers the power of $X$, and only for extensions
of $\mathbb{Q}$, would for example imply that every finite group
is a Galois group over $\mathbb{Q}$. As such, even this weak version
(let alone the full version) is naturally considered to be entirely
out of reach of current methods.

An upper bound at least as strong as that in Conjecture \ref{con:(Malle,weak)}
is known to hold in the following cases over general number fields
$K$:
\begin{enumerate}
\item For any abelian group \cite{Maki, Wright}, with the asymptotic constants
(in principle).
\item For any nilpotent group \cite{Kluners-Malle}. For a nilpotent group
in its regular representation, the lower bound is also known.
\item For $S_{3}$ \cite{Datskovsky-Wright,Davenport-Heilbronn}, with the
asymptotic constants. In fact, in this case there is a second main
term, and its asymptotic constant is also known \cite{Bhargava-Shankar-Tsimerman,Taniguchi-Thorne}.
\item For $D_{4}$ and $S_{4}$ \cite{Baily,Bhargava-3,Cohen-1,Bhargava-Shankar-Wang}.
The asymptotic constants are also known.  A power savings in the error term is also known \cite{Belabas-Bhargava-Pomerance} when $K=\mathbb{Q}$.
\item For $S_{5}$  \cite{Kable-Yukie,Bhargava-4, Bhargava-Shankar-Wang},
as well as the asymptotic constant.  A power savings in the error term is also known \cite{Shankar-Tsimerman} when $K=\mathbb{Q}$.
\item For degree-6 $S_{3}$ extensions \cite{Bhargava-Wood}, as well as
the asymptotic constant.
\item Under mild restrictions, for wreath products of the form $C_{2}\wr H$
where $H$ is nilpotent \cite{Kluners-2}.
\end{enumerate}
Note that the results in degree 4 provide a stark contrast for the
situation with counting polynomials by the maximum height of their
coefficients: if we let $a_{i}$ for $1\leq i\leq n$ be indeterminates,
then the polynomial $p(x)=x^{n}+a_{n-1}x^{n-1}+\cdots+a_{0}\in K(a_{1},\cdots,a_{n})$
has Galois group $S_{n}$ over $K(a_{1},\dots,a_{n})$. Hilbert's
Irreducibility Theorem then implies that almost all specializations
(when ordered by the coefficient height) of this polynomial still
have Galois group $S_{n}$.

However, the results of Cohen et al. \cite{Cohen-2} collectively show that, when
ordered by discriminant, a positive proportion (roughly 17\%) of extensions
of degree 4 have an associated Galois group isomorphic to the dihedral
group $D_{4}$: the difference is entirely caused by ordering the
fields by discriminant. Malle's conjectures, moreover, indicate that
the non-$S_{n}$ extensions should have a positive density for any
composite $n$, but should have zero density for prime $n$, though
this is not known to be true for any $n>5$.

\subsection{Outline of Results}

The overarching goal of this paper is to generalize the results of
Schmidt and Ellenberg-Venkatesh to arbitrary $G$-extensions. In Section
\ref{sec:Proof-of-permreps}, we prove the following theorem:

\begin{thmcust}{\ref{thm:Permreps}} Let $n\geq2$, let $K$ be
any number field, and let $G$ be a proper transitive subgroup of
$S_{n}$. Also, let $t$ be such that if $G'$ is the intersection
of any point stabilizer in $S_{n}$ with $G$, then any subgroup of
$G$ properly containing $G'$ has index at most $t$. Then for any
$\epsilon>0$, \[
N_{K,n}(X;G)\ll X^{\tfrac{1}{2(n-t)}\left[\sum_{i=1}^{n-1}\deg(f_{i+1})-\tfrac{1}{[K:\mathbb{Q}]}\right]+\epsilon},\]
where the $f_{i}$ for $1\leq i\leq n$ are a set of primary invariants
for $G$, whose degrees (in particular) satisfy $\deg(f_{i})\leq i$.
\end{thmcust}

We note here that for every primitive group covered by the Theorem,
the result is always strictly better than the result offered by Schmidt's
bound $N_{K,n}(X)\ll X^{(n+2)/4}$, and the savings
(see Appendix \ref{sec:Tabulation-of-Results}) are often significant.

Our proof follows the same general approach as that of Schmidt and
generalizes Example 2.7 from Ellenberg-Venkatesh \cite{EV-numberfields},
which gives a rough outline of the technique for a single group. The
technique is as follows:
\begin{enumerate}
\item Apply Minkowski's Theorems to obtain an algebraic integer generating
$L$ whose archimedean valuations are small.
\item Use a counting argument to establish an upper bound on the number
of such algebraic integers.
\end{enumerate}
The goal of Proposition \ref{lem:trace-zero} is to accomplish (1).
We modify the basic argument in (2) by rephrasing the counting argument
in scheme-theoretic language, and then invoke invariant theory and
the large sieve (see Lemma \ref{lem:Sieving-saving}) to save in the
counting part.

\section{Polynomial Invariants of Finite Groups\label{sec:Appendix-Invariants}}

In this section we briefly discuss some standard results in the theory
of polynomial invariants; we freely refer to results from this section
in the main text. The following discussion is condensed from Derksen-Kemper
\cite{Derksen-Kemper}.

Let $G$ be a finite group and $\rho:G\to GL_{n}(\mathbb{C})$ be
a (faithful) complex representation, and let $G$ act on $\mathbb{C}[x_{1},\cdots,x_{n}]$
via $\rho$. If $f_{1},\cdots,f_{n}$ are algebraically independent,
homogeneous elements of $\mathbb{C}[x_{1},\cdots,x_{n}]$ with the
property that $\mathbb{C}[x_{1},\cdots,x_{n}]^{G}$, the ring of $G$-invariant
polynomials, is a finitely-generated module over $\mathbb{C}[f_{1},\cdots,f_{n}]$,
we say these polynomials $f_{i}$ are a set of ``primary invariants''
for $G$. The Noether normalization lemma implies that such polynomials
exist; that there are $n$ of them follows from comparing transcendence
degrees.

The primary invariants are not unique: one can (for example) take
linear combinations or powers of the $f_{i}$ and still retain the
finite-generation property. When we speak of primary invariants, we
generally mean a set of primary invariants which are homogeneous and
of minimal degree, arranged in nondecreasing order by degree. However,
all results discussed will hold for any set of primary invariants.

Denote $A=\mathbb{C}[f_{1},\cdots,f_{n}]$, and $R=\mathbb{C}[x_{1},\cdots,x_{n}]^{G}$.
The theorem of Hochster-Roberts (Theorem 2.5.5 of \cite{Derksen-Kemper})
implies that $R$ is a Cohen-Macaulay ring and, moreover, that there
exist homogeneous $G$-invariant polynomials $g_{1},\, g_{2},\cdots,\, g_{k}$
with $g_{1}=1$ such that $R=A\cdot g_{1}+\cdots+A\cdot g_{k}$. These
polynomials $g_{i}$ are called ``secondary invariants'' of
$G$ and will depend intrinsically on the choice of primary invariants,
and are not uniquely determined even for a fixed set of primary invariants.
\begin{example*}
Let $G=S_{n}$ and $\rho$ be the regular representation of $G$ (which
acts by index permutation on $\mathbb{C}[x_{1},\cdots,x_{n}]$). It
is easy to see that the elementary symmetric polynomials are invariants
under the action of $G$ on $\mathbb{C}[x_{1},\cdots,x_{n}]$, and
that they are algebraically independent: thus, they form a set of
primary invariants for $G$. In fact, for \emph{any} subgroup of $S_{n}$,
the elementary symmetric polynomials form a set of (possibly non-minimal-degree)
primary invariants: hence, for any permutation representation $\rho$
of degree $n$, there exists a set of primary invariants of $\rho$
such that $\deg(f_{i})\leq i$ for each $1\leq i\leq n$.
\end{example*}
Associated to any (usually $G$-invariant) graded submodule $M$ of
$\mathbb{C}[x_{1},\cdots,x_{n}]$ is the generating function $H(M,t)={\displaystyle \sum_{j=0}^{\infty}a_{j}t^{j}}$,
where $a_{j}=\dim_{\mathbb{C}}(M^{(j)})$, the vector space dimension
of the degree-$j$ polynomials in $M$. This generating function is
called (variously) the Hilbert series or the Molien
series of $M$.
\begin{example*}
For $A=\mathbb{C}[f_{1},\cdots,f_{d}]$, one has $H(A,t)={\displaystyle \prod_{i=1}^{n}(1-t^{\deg(f_{i})})^{-1}}$
by the algebraic independence of the $f_{i}$.
\end{example*}
For $R=\mathbb{C}[x_{1},\cdots,x_{m}]^{G}$, there is a formula, due
to Molien, which says \begin{equation}
H(R,t)=\dfrac{1}{\left|G\right|}{\displaystyle \sum_{g\in G}\dfrac{1}{\det(I-t\rho(g))}}.\label{eq:Molien-1}\end{equation}
(In fact the formula applies to any linear representation $\rho\,:\, G\to GL(V)$,
over any field of characteristic relatively prime to $\left|G\right|$.)
By looking at the free resolution of $R=A\cdot g_{1}+\cdots+A\cdot g_{k}$
arising from the secondary invariants in tandem with \ref{eq:Molien-1},
we can write \begin{equation}
H(R,t)=\dfrac{1}{\left|G\right|}{\displaystyle \sum_{g\in G}\dfrac{1}{\det(I-t\rho(g))}}=\dfrac{\sum_{j=1}^{k}t^{\deg(g_{i})}}{\prod_{i=1}^{n}(1-t^{\deg(f_{i})})}.\label{eq:Molien}\end{equation}

By examining the Hilbert series identity \ref{eq:Molien} with sufficient
care, one can deduce a number of facts about the primary invariants:
for example, the product of the degrees of any set of primary invariants
is divisible by $\left|G\right|$, and the quotient is equal to the
number of associated secondary invariants (cf. Proposition 3.3.5 of
\cite{Derksen-Kemper}). Also, the least common multiple of the degrees
of the primary invariants is divisible by the exponent of $G$.

For any particular representation $\rho$, one can compute the Hilbert
series as a rational function using Molien's formula, and then factor
the denominator to generate possibilities for the degrees for the
primary invariants. One might hope that this will immediately give
the degrees of the primary invariants, but this is not the case: for
general linear representations (or even permutation representations),
the minimal degrees possible from the Hilbert series will not always
give the degrees of an actual set of primary invariants.

The computer algebra system MAGMA computes minimal primary invariants
by using Molien's formula to generate possible degree vectors for
the primary invariants, then generates independent $\rho$-invariant
polynomials of those degrees, and finally applies a Hilbert-driven
Buchberger algorithm to verify that the resulting ideal is zero-dimensional.
For a number of reasons, most algorithms for primary invariant computation
in general settings generally seek to minimize the product of the
invariant degrees rather than their sum. In general, we would also
not expect there to be a way to compute the degrees of a set of primary
invariants without essentially having to compute the invariants themselves;
see the discussion following Algorithm 3.3.4 of \cite{Derksen-Kemper}
for further details.

\section{Proof of Main Counting Theorem \label{sec:Proof-of-permreps}}

Given an extension $L/K$, we start by constructing a generator of
small size.
\begin{prop}
\label{lem:trace-zero} Let $K$ be a number field of degree $l$
over $\mathbb{Q}$, and $L/K$ an extension of degree $n$ such that
$\mathrm{Nm}_{K/\mathbb{Q}}(\mathcal{D}_{L/K})<X$, and such that
any proper subfield $K'$ of $L$ containing $K$ has $[K':K]\leq t$.
Then there exists an $\alpha\in\mathcal{O}_{L}$ with $\mathrm{Tr}_{L/K}(\alpha)=0$,
all of whose archimedean valuations have absolute value $\ll X^{\tfrac{1}{2l(n-t)}}$,
and such that $L=K(\alpha)$.\end{prop}
\begin{proof}
If $L$ has $r$ real embeddings $\rho_{1},\,\dots,\,\rho_{r}$ and
$s$ complex embeddings $\sigma_{1},\bar{\sigma}_{1},\,\dots,\,\sigma_{s},\bar{\sigma}_{s}$
(where $r+2s=nl$), for $\alpha\in L$ we define the {}``Minkowski
map'' $\varphi_{L}:L\to\mathbb{R}^{nl}=\mathbb{R}^{r+2s}$ sending
\[
\alpha\mapsto\left(\rho_{1}(\alpha),\,\dots,\rho_{r}(\alpha),\,\sqrt{2}\,\mathrm{Re}\,\sigma_{1}(\alpha),\,\sqrt{2}\,\mathrm{Im}\,\sigma_{1}(\alpha),\,\dots,\,\sqrt{2}\,\mathrm{Re}\,\sigma_{s}(\alpha),\,\sqrt{2}\,\mathrm{Im}\,\sigma_{s}(\alpha)\right).\]
Recall that the image $\Lambda_{L}=\varphi_{L}(\mathcal{O}_{L})$
is the so-called Minkowski lattice of rank $nl$ in $\mathbb{R}^{nl}$.

Let $\beta_{1},\cdots,\beta_{nl}$ be the successive minima of the
gauge function $f(x_{1},\cdots,x_{nl})=\max(x_{1},\dots,x_{nl})$
on $\Lambda_{L}$, and denote $f(\varphi(\beta_{i}))=\left|\left|\beta_{i}\right|\right|$
for shorthand. (Note that $\left|\left|\beta_{i}\right|\right|$ is
essentially just the maximum archimedean valuation of $\beta_{i}$
up to a factor of 2.) Minkowski's Second Theorem \cite{Siegel} says
\begin{equation}
\prod_{i=1}^{nl}\left|\left|\beta_{i}\right|\right|\ll\left|D_{L}\right|^{1/2},\label{eq:Minkowski-bound}\end{equation}
where the implied constant depends only on $nl$.

Now since the $\beta_{i}$ are nondecreasing, for any $k$ we may
use the bound given by \ref{eq:Minkowski-bound} to write \[
\left|\left|\beta_{k}\right|\right|^{nl+1-k}\leq\prod_{i=k}^{nl}\left|\left|\beta_{i}\right|\right|\leq\prod_{i=1}^{nl}\left|\left|\beta_{i}\right|\right|\ll D_{L}^{1/2}\]
whence \begin{equation}
\left|\left|\beta_{k}\right|\right|\ll D_{L}^{1/2(nl+1-k)}.\label{eq:norm-bound-1}\end{equation}

For all $k$ with $1\leq k\leq t+1$, \ref{eq:norm-bound-1} implies
\begin{equation}
\left|\left|\beta_{k}\right|\right|\ll D_{L}^{1/2l(n-t)}\ll X^{1/2l(n-t)}.\label{eq:norm-bound}\end{equation}

Now, by our assumption about intermediate subfields, we know that
$S=\{\beta_{1},\cdots,\beta_{t+1}\}$ will generate $L/K$, since
$S$ spans a vector subspace of $L$ of dimension greater than any
proper subfield. By a pigeonhole argument, we see that if $\mathrm{sub}(L/K)$
denotes the number of subfields of $L/K$ (which by Galois theory
can be bounded above in terms of $n$ only), there exists a linear
combination $\alpha_{1}=\sum_{S}c_{i}\beta_{i}$, with integral coefficients
bounded in absolute value by $\mathrm{sub}(L/K)$, that generates
$L/K$.

Since $K$ is fixed, we may choose a basis $B$ of $\mathcal{O}_{K}$
and observe that $S'=S\cup B$ still has the property that $\left|\left|\beta\right|\right|\ll X^{1/2l(n-t)}$
for every $\beta\in S'$. If $\pi$ is the projection of $\varphi(\left\langle S'\right\rangle )$
onto the sublattice of the Minkowski lattice generated by $B$, then
$\alpha=l\alpha_{1}-\pi(\alpha_{1})$ lies in $\mathcal{O}_{L}$,
has trace zero, generates $L/K$, and its archimedean norms satisfy
\begin{equation}
\left|\left|\alpha\right|\right|\ll X^{1/2l(n-t)}.\label{eq:norm-bound-2}\end{equation}

\end{proof}
We also require a sieving lemma: 
\begin{lem}
\label{lem:Sieving-saving} Suppose $\Pi:Z\mapsto\mathbb{A}^{d}$
is a finite map of schemes of degree $\geq2$ and $Z$ is irreducible.
Then, for any $\epsilon>0$, the number of integral points of $Z$
whose images lie in the box centered at 0 whose sides have lengths
$(X^{a_{1}},X^{a_{2}},\cdots,X^{a_{d}})$ is $\ll X^{(\sum a_{i})-\tfrac{1}{2}a_{1}+\epsilon}$,
where $a_{1}\leq a_{2}\leq\cdots\leq a_{d}$ are positive rational
numbers.\end{lem}
\begin{proof}
First, by changing variables for $X$, we may assume that the $a_{i}$
are integers. Our starting point is a multivariable version of Hilbert's
Irreducibility Theorem due to S.D. Cohen \cite{Cohen-Hilbert}: if
$X\to\mathbb{P}^{n}$ is a morphism of degree $\geq2$, then the number
of integral points of $\mathbb{A}^{n}$ of height $\leq N$ which
lift to $X$ is $\ll N^{n-1/2+\epsilon}$.

The side length of the box in that theorem is $N$, and the result
gives a savings of $N^{1/2-\epsilon}$ on the box. The result is also
stated for a box centered at 0, but the bound (with a uniform constant)
still holds even if we translate to center the box at an arbitrary
point.

Now we tile our large box of side lengths $(X^{a_{1}},X^{a_{2}},\cdots,X^{a_{d}})$
with square boxes each of which has size $(X^{a_{1}},X^{a_{1}},\cdots,X^{a_{1}})$:
each square box yields $\ll X^{da_{1}-\frac{1}{2}a_{1}+\epsilon}$
points of $Z$ having an image in that square box, and we require
a total of $X^{(\sum a_{i})-da_{1}}$ such square boxes to cover the
large box. The result follows.\end{proof}
\begin{rem*}
There are sieving methods that work directly with non-square boxes,
and these would presumably give an additional small savings, although
we do not expect the gain to be particularly significant.
\end{rem*}
We can now prove the main theorem: 
\begin{thm}
\label{thm:Permreps} Let $n\geq2$, let $K$ be any number field,
and let $G$ be a proper transitive subgroup of $S_{n}$. Also, let
$t$ be such that if $G'$ is the intersection of a point stabilizer
in $S_{n}$ with $G$, then any subgroup of $G$ properly containing
$G'$ has index at most $t$. Then for any $\epsilon>0$, \begin{equation}
N_{K,n}(X;G)\ll X^{\tfrac{1}{2(n-t)}\left[\sum_{i=1}^{n-1}\deg(f_{i+1})-\tfrac{1}{[K:\mathbb{Q}]}\right]+\epsilon},\label{eq:permreps-bound}\end{equation}
where the $f_{i}$ for $1\leq i\leq n$ are a set of primary invariants
for $G$, whose degrees (in particular) satisfy $\deg(f_{i})\leq i$.\end{thm}
\begin{rem*}
\label{rem:Iintermediate-extension-criterion} The condition about
the point stabilizer is (by the Galois correspondence) equivalent
to the following: if $L/K$ is a $G$-extension, then any proper subfield
$K'$ of $L$ containing $K$ has $[K':K]\leq t$. (The criterion
in the theorem statement is stated the way it is in order to avoid
any reference to $L$.) We note in particular that if $G$ is a primitive
subgroup of $S_{n}$, then $t=1$.\end{rem*}
\begin{proof}
Let $G$ act on the polynomial ring $\mathbb{C}[x_{1},\cdots,x_{n}]$
by index permutation, and let $f_{1},\cdots,f_{n}$ be primary invariants
of $G$ with associated secondary invariants $1=g_{1},g_{2},\cdots,g_{k}$,
each set arranged in order of nondecreasing degree. Observe that because
$G$ is transitive, the only primary invariant of degree 1 is $f_{1}=x_{1}+\cdots+x_{n}$,
and that because $G$ is proper, there is at least one secondary invariant
besides $g_{1}=1$.

Denote $A=\mathbb{C}[f_{1},\cdots,f_{n}]$ and $R=\mathbb{C}[x_{1},\cdots,x_{n}]^{G}$,
and observe that $\bar{R}=R/f_{1}R$ is an integral domain. Let $S$
be the subring of $\bar{R}$ generated by $\bar{f_{2}},\cdots,\bar{f_{n}}$
and $\bar{g}_{2}$, and let $Z=\mathrm{Spec}(S)$. Observe that $S$
is an integral domain (since $\bar{R}$ is) so $Z$ is irreducible.

The natural map $\mathbb{C}[f_{2},\cdots,f_{n}]\to S$ induces a projection
$\Pi:Z\to\mathbb{A}^{n-1}$ (namely, evaluation of the polynomials
$f_{2},\dots,\, f_{n}$ at the given point), and the map $\Pi$ is
finite because $R$ is a finitely-generated $A$-module (whence $\bar{R}$
is finite over $\mathbb{C}[f_{2},\cdots,f_{n}]$). Also notice that,
by construction, we have $\bar{g}_{2}\not\in\mathbb{C}[\bar{f}_{2},\cdots,\bar{f}_{n}]$,
and so $\Pi$ has degree at least 2.

Now suppose $L/K$ is an extension of number fields with $[K:\mathbb{Q}]=l$,
$[L:K]=n$, such that the Galois group of the Galois closure $\hat{L}/K$
is permutation-isomorphic to $G$, and such that $\mathrm{Nm}_{K/\mathbb{Q}}(\mathcal{D}_{L/K})<X$.
As noted in Remark \ref{rem:Iintermediate-extension-criterion}, the
condition on the group $G$ implies that any field $K'$ intermediate
between $K$ and $L$ has $[K':K]\leq t$. By Proposition \ref{lem:trace-zero},
there exists a nonzero element $\alpha\in\mathcal{O}_{L}$ of trace
zero such that all archimedean valuations of $\alpha$ are $\ll X^{\tfrac{1}{2l(n-t)}}$
and with $L=K(\alpha)$. This element $\alpha$ gives rise to an integral
point ${\bf x}=(\alpha^{(1)},\dots,\alpha^{(n)})\in Z$, where the
$\alpha^{(i)}$ are the archimedean embeddings of $\alpha$. (Note
that we are using the fact that $\alpha$ has trace zero to say that
$f_{1}({\bf x})=0$, so that ${\bf x}$ is actually well-defined on
$Z$.)

We may then obtain an upper bound on the total possible number of
fields $L$ by bounding the number of possible ${\bf x}$. But since
$\Pi$ is finite (and its degree is independent of $L$), we may equivalently
bound the number of possibilities for $\Pi({\bf x})$.

Since $\Pi$ is simply evaluation of the primary invariant polynomials
$f_{i}$ on the point ${\bf x}$, the coordinates of $\Pi(\mathbf{x})=(y_{2},\cdots,y_{n})$
obey the bounds \[
\left|y_{i}\right|\ll X^{\tfrac{\deg(f_{i})}{2l(n-t)}},\]
for $2\leq i\leq n$, which forms a {}``box'' $B$ in $\mathbb{A}^{n}(K)$.
By choosing an integral basis of $\mathcal{O}_{K}$, this box becomes
a box in $\mathbb{A}^{nl}(\mathbb{Q})$ with the same dimensions (up
to fixed constants), each occurring $l$ times, and the image of $\Pi({\bf x})$
is integral. We now apply Lemma \ref{lem:Sieving-saving} to see that
the number of possible integral points ${\bf x}$ is $\ll X^{\tfrac{1}{2(n-t)l}\left[l\sum_{i=1}^{n-1}\deg(f_{i+1})-\tfrac{1}{2}\deg(f_{2})\right]+\epsilon}$.
Finally, since $\deg(f_{2})=2$ and each ${\bf x}$ gives rise to
at most one distinct extension $L/K$, we obtain \[
N_{K,n}(X;G)\leq\#\{\text{integral }{\bf x}\in Z\ \text{with }\Pi({\bf x})\in B\}\ll X^{\tfrac{1}{2(n-t)}\left[\sum_{i=1}^{n-1}\deg(f_{i+1})-\tfrac{1}{l}\right]+\epsilon},\]
which is precisely the desired result.\end{proof}
\begin{rem*}
Note that we require the existence of a secondary invariant in order
to apply Lemma \ref{lem:Sieving-saving}. Without a secondary invariant,
we lose the power savings and instead obtain the upper bound $X^{\tfrac{1}{2(n-t)}\left[\sum_{i=1}^{n-1}\deg(f_{i+1})\right]}$.
This will only occur when $G=S_{n}$, whose primary invariants are
the usual symmetric polynomials (with degrees $2,3,\cdots,n$): it
is then easy to see that our upper bound is \[
X^{\tfrac{1}{2(n-1)}[\sum_{i=1}^{n-1}(i+1)]}=X^{\tfrac{1}{2(n-1)}[n(n+1)/2-1]}=X^{\tfrac{n+2}{4}},\]
which is precisely Schmidt's bound. Since the symmetric polynomials
are a set of primary invariants for any permutation group, we therefore
see that for any primitive proper transitive subgroup of $S_{n}$,
our theorem always beats the bound of Schmidt (due to the power-savings
from the sieving and the fact that $t=1$). However, in practice for
most primitive groups $G$, the majority of the actual savings comes
from the primary invariants, whose degrees tend to be much smaller
than the degrees of the symmetric polynomials.
\end{rem*}

\section{A Prototypical Example: $PSL_{2}(\mathbb{F}_{7})$ in $S_{7}$\label{sec:A-Generic-Example}}

In this section we give an explicit example of a primary invariant
computation, for the group $G=PSL_{2}(\mathbb{F}_{7})\cong GL_{3}(\mathbb{F}_{2})$,
which is the simple group of order 168, and appears as 7T5 in the
tables in Appendix \ref{sec:Tabulation-of-Results}.
\begin{cor}
For any $\epsilon>0$, $N_{\mathbb{Q},7}(X;G)\ll X^{11/6+\epsilon}.$
\end{cor}
For comparison, Schmidt's bound (for general degree-7 extensions)
gives an upper bound of $X^{9/4}$, and the Ellenberg-Venkatesh bound
is weaker.
\begin{proof}
Let $G=\left\langle (1\,2\,3\,4\,5\,6\,7),\ (1\,2)(3\,6)\right\rangle $;
it is a primitive permutation group on $\left\{ 1,2,3,4,5,6,7\right\} $
whose action is conjugate to the action of $PSL_{2}(\mathbb{F}_{7})$
on $\mathbb{P}^{1}(\mathbb{F}_{7})$. A computation with MAGMA shows
that primary invariants can be chosen as
\begin{eqnarray*}
f_{1} & = & x_{1}+x_{2}+x_{3}+x_{4}+x_{5}+x_{6}+x_{7}\\
f_{2} & = & x_{1}^{2}+x_{2}^{2}+x_{3}^{2}+x_{4}^{2}+x_{5}^{2}+x_{6}^{2}+x_{7}^{2}\\
f_{3} & = & x_{1}^{3}+x_{2}^{3}+x_{3}^{3}+x_{4}^{3}+x_{5}^{3}+x_{6}^{3}+x_{7}^{3}\\
f_{4} & = & x_{1}x_{2}x_{3}+x_{1}x_{2}x_{5}+x_{1}x_{2}x_{6}+x_{1}x_{2}x_{7}+x_{1}x_{3}x_{4}+x_{1}x_{3}x_{6}+x_{1}x_{3}x_{7}+x_{1}x_{4}x_{5}\\
 &  & +x_{1}x_{4}x_{6}+x_{1}x_{4}x_{7}+x_{1}x_{5}x_{6}+x_{1}x_{5}x_{7}+x_{2}x_{3}x_{4}+x_{2}x_{3}x_{5}+x_{2}x_{3}x_{7}+x_{2}x_{4}x_{5}\\
 &  & +x_{2}x_{4}x_{6}+x_{2}x_{4}x_{7}+x_{2}x_{5}x_{6}+x_{2}x_{6}x_{7}+x_{3}x_{4}x_{5}+x_{3}x_{4}x_{6}+x_{3}x_{5}x_{6}+x_{3}x_{5}x_{7}\\
 &  & +x_{3}x_{6}x_{7}+x_{4}x_{5}x_{7}+x_{4}x_{6}x_{7}+x_{5}x_{6}x_{7}\\
f_{5} & = & x_{1}^{4}+x_{2}^{4}+x_{3}^{4}+x_{4}^{4}+x_{5}^{4}+x_{6}^{4}+x_{7}^{4}\\
f_{6} & = & x_{1}^{2}x_{2}x_{3}+x_{1}^{2}x_{2}x_{5}+x_{1}^{2}x_{2}x_{6}+x_{1}^{2}x_{2}x_{7}+x_{1}^{2}x_{3}x_{4}+x_{1}^{2}x_{3}x_{6}+x_{1}^{2}x_{3}x_{7}+x_{1}^{2}x_{4}x_{5}\\
 &  & +x_{1}^{2}x_{4}x_{6}+x_{1}^{2}x_{4}x_{7}+x_{1}^{2}x_{5}x_{6}+x_{1}^{2}x_{5}x_{7}+x_{1}x_{2}^{2}x_{3}+x_{1}x_{2}^{2}x_{5}+x_{1}x_{2}^{2}x_{6}+x_{1}x_{2}^{2}x_{7}\\
 &  & +x_{1}x_{2}x_{3}^{2}+x_{1}x_{2}x_{5}^{2}+x_{1}x_{2}x_{6}^{2}+x_{1}x_{2}x_{7}^{2}+x_{1}x_{3}^{2}x_{4}+x_{1}x_{3}^{2}x_{6}+x_{1}x_{3}^{2}x_{7}+x_{1}x_{3}x_{4}^{2}\\
 &  & +x_{1}x_{3}x_{6}^{2}+x_{1}x_{3}x_{7}^{2}+x_{1}x_{4}^{2}x_{5}+x_{1}x_{4}^{2}x_{6}+x_{1}x_{4}^{2}x_{7}+x_{1}x_{4}x_{5}^{2}+x_{1}x_{4}x_{6}^{2}+x_{1}x_{4}x_{7}^{2}\\
 &  & +x_{1}x_{5}^{2}x_{6}+x_{1}x_{5}^{2}x_{7}+x_{1}x_{5}x_{6}^{2}+x_{1}x_{5}x_{7}^{2}+x_{2}^{2}x_{3}x_{4}+x_{2}^{2}x_{3}x_{5}+x_{2}^{2}x_{3}x_{7}+x_{2}^{2}x_{4}x_{5}\\
 &  & +x_{2}^{2}x_{4}x_{6}+x_{2}^{2}x_{4}x_{7}+x_{2}^{2}x_{5}x_{6}+x_{2}^{2}x_{6}x_{7}+x_{2}x_{3}^{2}x_{4}+x_{2}x_{3}^{2}x_{5}+x_{2}x_{3}^{2}x_{7}+x_{2}x_{3}x_{4}^{2}\\
 &  & +x_{2}x_{3}x_{5}^{2}+x_{2}x_{3}x_{7}^{2}+x_{2}x_{4}^{2}x_{5}+x_{2}x_{4}^{2}x_{6}+x_{2}x_{4}^{2}x_{7}+x_{2}x_{4}x_{5}^{2}+x_{2}x_{4}x_{6}^{2}+x_{2}x_{4}x_{7}^{2}\\
 &  & +x_{2}x_{5}^{2}x_{6}+x_{2}x_{5}x_{6}^{2}+x_{2}x_{6}^{2}x_{7}+x_{2}x_{6}x_{7}^{2}+x_{3}^{2}x_{4}x_{5}+x_{3}^{2}x_{4}x_{6}+x_{3}^{2}x_{5}x_{6}+x_{3}^{2}x_{5}x_{7}\\
 &  & +x_{3}^{2}x_{6}x_{7}+x_{3}x_{4}^{2}x_{5}+x_{3}x_{4}^{2}x_{6}+x_{3}x_{4}x_{5}^{2}+x_{3}x_{4}x_{6}^{2}+x_{3}x_{5}^{2}x_{6}+x_{3}x_{5}^{2}x_{7}+x_{3}x_{5}x_{6}^{2}\\
 &  & +x_{3}x_{5}x_{7}^{2}+x_{3}x_{6}^{2}x_{7}+x_{3}x_{6}x_{7}^{2}+x_{4}^{2}x_{5}x_{7}+x_{4}^{2}x_{6}x_{7}+x_{4}x_{5}^{2}x_{7}+x_{4}x_{5}x_{7}^{2}+x_{4}x_{6}^{2}x_{7}\\
 &  & +x_{4}x_{6}x_{7}^{2}+x_{5}^{2}x_{6}x_{7}+x_{5}x_{6}^{2}x_{7}+x_{5}x_{6}x_{7}^{2}\\
f_{7} & = & x_{1}^{7}+x_{2}^{7}+x_{3}^{7}+x_{4}^{7}+x_{5}^{7}+x_{6}^{7}+x_{7}^{7}\end{eqnarray*}
of degrees $1,\,2,\,3,\,3,\,4,\,4,\,7$ respectively. Invoking Theorem
\ref{thm:Permreps} yields the stated bound. (Note here that $t=1$.)
\end{proof}
We will note that the group $G=PSL_{2}(\mathbb{F}_{7})$ also appears
as a transitive subgroup of $S_{8}$ (it is 8T37 in the tables in
Appendix \ref{sec:Tabulation-of-Results}), but the upper bounds obtained
are different: as a subgroup of $S_{7}$ we obtain the bound $X^{11/6+\epsilon}$,
while as a subgroup of $S_{8}$ we obtain $X^{29/14+\epsilon}$. This
should not be surprising, as the fields being counted are different
(though related): in the $S_{7}$ case we are counting fields of degree
7 whose Galois action on the 7 complex embeddings is that of $G$,
whereas in the $S_{8}$ case we are counting fields of degree 8 whose
Galois action on the 8 complex embeddings is that of $G$.  Indeed, the predictions from Malle's heuristics also differ for these fields: the number of degree-7 $G$-extensions is expected to be approximately $X^{1/2 + \epsilon}$ while the number of degree-8 $G$-extensions is expected to be approximately $X^{1/4 + \epsilon}$.

\section{Closing Remarks}

Per Malle's heuristics, we would expect
the actual number of integral points to be (much) lower than the bound
given by Theorem \ref{thm:Permreps}. There are three ways in which
we lose accuracy:
\begin{enumerate}
\item The map associating an element ${\bf x}$ to an extension $L/K$ is
not injective: any extension has many different generators. Worse
still, there is no uniform way to account for this non-injectivity:
an extension of small discriminant will have many generators of small
archimedean norm, and thus it will show up in the count much more
frequently than an extension of larger discriminant.
\item The simple techniques employed above for counting integral points
on the scheme $Z$ give weaker bounds than could be hoped for. Most
points in affine space are not actually the image of an integral point
on $Z$, so we would not expect that the sieving lemma \ref{lem:Sieving-saving}
is sharp: it is likely only extracting a small amount of the potential
savings that should be realizable.
\item If $L/K$ has any intermediate extensions, the bound given in Lemma \ref{lem:trace-zero}
on the archimedean norm of a generator is weaker than for a primitive extension. The worst losses occur when
$L/K$ has a subfield of small index (e.g., index 2), in which case
the exponent obtained in Theorem \ref{thm:Permreps} is nearly doubled.
\end{enumerate}
One technique by which we could address the issues in (1) is that
of Ellenberg-Venkatesh \cite{EV-numberfields}: rather than counting
the number of possibilities for the single element ${\bf x}$ of trace
zero and whose archimedean valuations are small, we could instead
count the number of possibilities for an $r$-tuple of elements $({\bf x}_{1},\cdots,{\bf x}_{r})$,
each of whose archimedean valuations is small. This would provide
a stronger way of separating extensions of differing discriminants
and reduce the amount of duplication in the counting (though it cannot
entirely erase duplicate counting).

In order to address the deficiencies of (2), we would require the
use of stronger point-counting techniques. To do this, however, would require 
understanding the geometry of the scheme $Z$ in a much
deeper way. For particular groups $G$ with low-degree permutation
representations, this is (at least, theoretically) feasible, since
the primary invariants are explicitly computable. However, for large
$n$ this seems very unlikely to succeed, since the invariant theory
becomes extremely computationally demanding for $n>10$.

To deal with the deficiencies of (3), it seems likely that a more
direct analysis of the possible extension towers for extensions of
small degree over general base fields could yield significant savings,
but we will not pursue this avenue here.

As a concluding remark, one way of reinterpreting Theorem \ref{thm:Permreps} is to view it
as a result about permutation representations of groups. The invariant
theory involved in the proof carries over to general representations
$\rho$, and so one could ask: is there a way to construct an analogue of these results 
attached to an arbitrary faithful representation $\rho$? As we show
in a forthcoming paper \cite{Dummit-rhodisc}, the answer to this question is also {}``yes''.

\appendix

\section{Tabulation of Results\label{sec:Tabulation-of-Results}}

In the following tables, we give the results of the invariant computations,
performed using the algebra system MAGMA, for all proper transitive
subgroups of $S_{n}$ for $n=5,6,7,8$, along with a small number
of subgroups of $S_{9}$ for which it was possible to finish the invariant
computations within 2 days on a 4Ghz desktop computer with 1GB of
memory. We observe that for primitive transitive subgroups, the result
of Theorem \ref{thm:Permreps} is significantly better than the overall bound of Schmidt,
although the results generally do not get close to $X^{1}$ nor (a
fortiori) to the bounds in Malle's Conjecture \ref{con:(Malle,weak)}.
For imprimitive extensions, and especially in even degree (where many
extensions have an index-2 subfield), the results are frequently worse
than Schmidt's bound.

The labeling of the transitive subgroups is the standard one originally
given by Conway-Hulpke-McKay \cite{Conway-transitivegroups}. Subfield
information was obtained from John Jones' page on transitive group
data \cite{Jones-transitivegroups}, which also contains additional
detailed information about the transitive subgroups.

For brevity in the tables below, we quote the results of Theorem \ref{thm:Permreps}
only for the base field $K=\mathbb{Q}$, and we write the results
as $X^{\#}$ rather than $X^{\#+\epsilon}$ (including the bounds
conjectured by Malle). The upper bound over a general base field $K$
of degree $l$ over $\mathbb{Q}$ is (for an entry of $X^{\#}$) equal
to $X^{\#+1-\tfrac{1}{l}+\epsilon}$. Rows marked with an asterisk
are groups for which Malle's weak conjecture is known to hold. For
subgroups of $S_{5}$, we compare the results to the bound of Bhargava;
for other symmetric groups, we compare our results to that of Schmidt.

We also remark that for certain classes of groups such as the dihedral
groups, there are bounds available (e.g., from class field theory)
that are far better than Schmidt's bound.

\begin{longtable}{|c|c|c|c|c|c|c|c|}
\hline 
\multicolumn{8}{|c|}{Proper transitive subgroups of $S_{5}$}\tabularnewline
\hline 
\# & Order & Isom. to & Subfield? & Invariant Degrees & Result & Malle & Bhargava\tabularnewline
\hline 
\endfirsthead
\hline
\multicolumn{8}{|c|}{Proper transitive subgroups of $S_{5}$ (continued)}\tabularnewline
\hline 
\# & Order & Isom. to & Subfield? & Invariant Degrees & Result & Malle & Bhargava\tabularnewline
\hline 
\endhead
5T1 & 5{*} & $C_{5}$ & none & 1,2,2,3,5 & $X^{11/8}$ & $X^{1/4}$ & $X^{1}$\tabularnewline
\hline 
5T2 & 10 & $D_{5}$ & none & 1,2,2,3,5 & $X^{11/8}$ & $X^{1/2}$ & $X^{1}$\tabularnewline
\hline 
5T3 & 20 & $F_{20}$ & none & 1,2,3,4,5 & $X^{13/8}$ & $X^{1/2}$ & $X^{1}$\tabularnewline
\hline 
5T4 & 60 & $A_{5}$ & none & 1,2,3,4,5 & $X^{13/8}$ & $X^{1/2}$ & $X^{1}$\tabularnewline
\hline
\end{longtable}

\begin{longtable}{|c|c|c|c|c|c|c|c|}
\hline 
\multicolumn{8}{|c|}{Proper transitive subgroups of $S_{6}$}\tabularnewline
\hline 
\# & Ord & Isom. to & Subfield? & Invariant Degrees & Result & Malle & Schmidt\tabularnewline
\hline 
\endfirsthead
\hline
\multicolumn{8}{|c|}{Proper transitive subgroups of $S_{6}$ (continued)}\tabularnewline
\hline 
\# & Ord & Isom. to & Subfield? & Invariant Degrees & Result & Malle & Schmidt\tabularnewline
\hline
\endhead
6T1 & 6{*} & $C_{6}$ & Deg. 3 & 1,2,2,2,3,6 & $X^{7/3}$ & $X^{1/3}$ & $X^{2}$\tabularnewline
\hline 
6T2 & 6{*} & $S_{3}$ & Deg. 3 & 1,2,2,2,3,3 & $X^{11/6}$ & $X^{1/3}$ & $X^{2}$\tabularnewline
\hline 
6T3 & 12 & $S_{3}\times C_{2}$ & Deg. 3 & 1,2,2,2,3,6 & $X^{7/3}$ & $X^{1/2}$ & $X^{2}$\tabularnewline
\hline 
6T4 & 12 & $A_{4}$ & Deg. 3 & 1,2,2,3,3,4 & $X^{2}$ & $X^{1/2}$ & $X^{2}$\tabularnewline
\hline 
6T5 & 18 & $F_{18}$ & Deg. 2 & 1,2,2,3,3,6 & $X^{7/4}$ & $X^{1/2}$ & $X^{2}$\tabularnewline
\hline 
6T6 & 24 & $A_{4}\times C_{2}$ & Deg. 3 & 1,2,2,3,4,6 & $X^{8/3}$ & $X^{1}$ & $X^{2}$\tabularnewline
\hline 
6T7 & 24 & $S_{4}$ & Deg. 3 & 1,2,2,3,3,4 & $X^{13/6}$ & $X^{1/2}$ & $X^{2}$\tabularnewline
\hline 
6T8 & 24 & $S_{4}$ & Deg. 3 & 1,2,2,3,4,6 & $X^{8/3}$ & $X^{1/2}$ & $X^{2}$\tabularnewline
\hline 
6T9 & 36 & $S_{3}\times S_{3}$ & Deg. 2 & 1,2,2,3,4,6 & $X^{2}$ & $X^{1/2}$ & $X^{2}$\tabularnewline
\hline 
6T10 & 36 & $F_{36}$ & Deg. 2 & 1,2,3,3,4,6 & $X^{17/8}$ & $X^{1/2}$ & $X^{2}$\tabularnewline
\hline 
6T11 & 48 & $S_{4}\times C_{2}$ & Deg. 3 & 1,2,2,3,4,6 & $X^{8/3}$ & $X^{1}$ & $X^{2}$\tabularnewline
\hline 
6T12 & 60 & $A_{5}$ & none & 1,2,3,3,4,5 & $X^{8/5}$ & $X^{1/2}$ & $X^{2}$\tabularnewline
\hline 
6T13 & 72 & $F_{36}\rtimes C_{2}$ & Deg. 2 & 1,2,2,3,4,6 & $X^{2}$ & $X^{1}$ & $X^{2}$\tabularnewline
\hline 
6T14 & 120 & $S_{5}$ & none & 1,2,3,4,5,6 & $X^{19/10}$ & $X^{1/2}$ & $X^{2}$\tabularnewline
\hline 
6T15 & 360 & $A_{6}$ & none & 1,2,3,4,5,6 & $X^{19/10}$ & $X^{1/2}$ & $X^{2}$\tabularnewline
\hline
\end{longtable} 

\begin{longtable}{|c|c|c|c|c|c|c|c|}
\hline 
\multicolumn{8}{|c|}{Proper transitive subgroups of $S_{7}$}\tabularnewline
\hline 
\# & Order & Isom. to & Subfield? & Invariant Degrees & Result & Malle & Schmidt\tabularnewline
\hline 
\endfirsthead
\hline
\multicolumn{8}{|c|}{Proper transitive subgroups of $S_{7}$ (continued)}\tabularnewline
\hline 
\# & Ord & Isom. to & Subfield? & Invariant Degrees & Result & Malle & Schmidt\tabularnewline
\hline 
\endhead
7T1 & 7{*} & $C_{7}$ & none & 1,2,2,2,3,4,7 & $X^{19/12}$ & $X^{1/6}$ & $X^{9/4}$\tabularnewline
\hline 
7T2 & 14 & $D_{7}$ & none & 1,2,2,2,3,4,7 & $X^{19/12}$ & $X^{1/3}$ & $X^{9/4}$\tabularnewline
\hline 
7T3 & 21 & $F_{21}$ & none & 1,2,3,3,3,4,7 & $X^{7/4}$ & $X^{1/4}$ & $X^{9/4}$\tabularnewline
\hline 
7T4 & 42 & $F_{42}$ & none & 1,2,3,3,4,6,7 & $X^{2}$ & $X^{1/3}$ & $X^{9/4}$\tabularnewline
\hline 
7T5 & 168 & $PSL_{2}(\mathbb{F}_{7})$ & none & 1,2,3,3,4,4,7 & $X^{11/6}$ & $X^{1/2}$ & $X^{9/4}$\tabularnewline
\hline 
7T6 & 2520 & $A_{7}$ & none & 1,2,3,4,5,6,7 & $X^{13/6}$ & $X^{1/2}$ & $X^{9/4}$\tabularnewline
\hline
\end{longtable}

\begin{longtable}{|c|c|c|c|c|c|c|c|}
\hline 
\multicolumn{8}{|c|}{Proper transitive subgroups of $S_{8}$}\tabularnewline
\hline 
\# & Order & Isom. to & Subfield? & Invariant Degrees & Result & Malle & Schmidt\tabularnewline
\hline 
\endfirsthead
\hline
\multicolumn{8}{|c|}{Proper transitive subgroups of $S_{8}$ (continued)}\tabularnewline
\hline 
\# & Order & Isom. to & Subfield? & Invariant Degrees & Result & Malle & Schmidt\tabularnewline
\hline 
\endhead
8T1 & 8{*} & $C_{8}$ & Deg. 4 & 1,2,2,2,2,3,4,8 & $X^{11/4}$ & $X^{1/4}$ & $X^{5/2}$\tabularnewline
\hline 
8T2 & 8{*} & $C_{4}\times C_{2}$ & Deg. 4 & 1,2,2,2,2,2,4,4 & $X^{17/8}$ & $X^{1/4}$ & $X^{5/2}$\tabularnewline
\hline 
8T3 & 8{*} & $(C_{2})^{3}$ & Deg. 4 & 1,2,2,2,2,2,2,2 & $X^{13/8}$ & $X^{1/4}$ & $X^{5/2}$\tabularnewline
\hline 
8T4 & 8{*} & $D_{4}$ & Deg. 4 & 1,2,2,2,2,2,4,4 & $X^{17/8}$ & $X^{1/4}$ & $X^{5/2}$\tabularnewline
\hline 
8T5 & 8{*} & $Q_{8}$ & Deg. 4 & 1,2,2,2,2,4,4,4 & $X^{19/8}$ & $X^{1/4}$ & $X^{5/2}$\tabularnewline
\hline 
8T6 & 16{*} &  & Deg. 4 & 1,2,2,2,2,3,4,8 & $X^{11/4}$ & $X^{1/3}$ & $X^{5/2}$\tabularnewline
\hline 
8T7 & 16{*} &  & Deg. 4 & 1,2,2,2,3,4,4,8 & $X^{3}$ & $X^{1/2}$ & $X^{5/2}$\tabularnewline
\hline 
8T8 & 16{*} &  & Deg. 4 & 1,2,2,2,3,4,4,8 & $X^{3}$ & $X^{1/3}$ & $X^{5/2}$\tabularnewline
\hline 
8T9 & 16{*} & $D_{4}\rtimes C_{2}$ & Deg. 4 & 1,2,2,2,2,2,4,4 & $X^{17/8}$ & $X^{1/2}$ & $X^{5/2}$\tabularnewline
\hline 
8T10 & 16{*} &  & Deg. 4 & 1,2,2,2,2,3,4,4 & $X^{9/4}$ & $X^{1/2}$ & $X^{5/2}$\tabularnewline
\hline
8T11 & 16{*} &  & Deg. 4 & 1,2,2,2,2,4,4,4 & $X^{19/8}$ & $X^{1/2}$ & $X^{5/2}$\tabularnewline
\hline 
8T12 & 24 & $SL_{2}(\mathbb{F}_{3})$ & Deg. 4 & 1,2,2,3,3,4,4,6 & $X^{23/8}$ & $X^{1/4}$ & $X^{5/2}$\tabularnewline
\hline 
8T13 & 24 & $A_{4}\times C_{2}$ & Deg. 4 & 1,2,2,2,3,3,4,6 & $X^{21/8}$ & $X^{1/4}$ & $X^{5/2}$\tabularnewline
\hline 
8T14 & 24 & $S_{4}$ & Deg. 4 & 1,2,2,2,3,4,4,6 & $X^{11/4}$ & $X^{1/4}$ & $X^{5/2}$\tabularnewline
\hline 
8T15 & 32{*} &  & Deg. 4 & 1,2,2,2,3,4,4,8 & $X^{3}$ & $X^{1/2}$ & $X^{5/2}$\tabularnewline
\hline 
8T16 & 32{*} &  & Deg. 4 & 1,2,2,2,3,4,4,8 & $X^{3}$ & $X^{1/2}$ & $X^{5/2}$\tabularnewline
\hline 
8T17 & 32{*} &  & Deg. 4 & 1,2,2,2,3,4,4,8 & $X^{3}$ & $X^{1/2}$ & $X^{5/2}$\tabularnewline
\hline 
8T18 & 32{*} &  & Deg. 4 & 1,2,2,2,2,3,4,4 & $X^{9/4}$ & $X^{1/2}$ & $X^{5/2}$\tabularnewline
\hline 
8T19 & 32{*} &  & Deg. 4 & 1,2,2,2,3,4,4,4 & $X^{5/2}$ & $X^{1/2}$ & $X^{5/2}$\tabularnewline
\hline 
8T20 & 32{*} &  & Deg. 4 & 1,2,2,2,3,4,4,4 & $X^{5/2}$ & $X^{1/2}$ & $X^{5/2}$\tabularnewline
\hline 
8T21 & 32{*} &  & Deg. 4 & 1,2,2,2,2,4,4,4 & $X^{19/8}$ & $X^{1/2}$ & $X^{5/2}$\tabularnewline
\hline 
8T22 & 32{*} &  & Deg. 4 & 1,2,2,2,2,4,4,4 & $X^{19/8}$ & $X^{1/2}$ & $X^{5/2}$\tabularnewline
\hline 
8T23 & 48 & $GL_{2}(\mathbb{F}_{3})$ & Deg. 4 & 1,2,2,3,3,4,6,8 & $X^{27/8}$ & $X^{1/3}$ & $X^{5/2}$\tabularnewline
\hline 
8T24 & 48 & $S_{4}\times C_{2}$ & Deg. 4 & 1,2,2,2,3,4,4,6 & $X^{11/4}$ & $X^{1/2}$ & $X^{5/2}$\tabularnewline
\hline 
8T25 & 56 & $F_{56}$ & none & 1,2,3,4,4,4,4,7 & $X^{27/14}$ & $X^{1/4}$ & $X^{5/2}$\tabularnewline
\hline 
8T26 & 64{*} &  & Deg. 4 & 1,2,2,2,3,4,4,8 & $X^{3}$ & $X^{1/2}$ & $X^{5/2}$\tabularnewline
\hline 
8T27 & 64{*} &  & Deg. 4 & 1,2,2,2,3,4,4,8 & $X^{3}$ & $X^{1}$ & $X^{5/2}$\tabularnewline
\hline 
8T28 & 64{*} &  & Deg. 4 & 1,2,2,2,3,4,4,8 & $X^{3}$ & $X^{1/2}$ & $X^{5/2}$\tabularnewline
\hline 
8T29 & 64{*} &  & Deg. 4 & 1,2,2,2,3,4,4,4 & $X^{5/2}$ & $X^{1/2}$ & $X^{5/2}$\tabularnewline
\hline 
8T30 & 64{*} &  & Deg. 4 & 1,2,2,2,3,4,4,8 & $X^{3}$ & $X^{1/2}$ & $X^{5/2}$\tabularnewline
\hline
8T31 & 64{*} &  & Deg. 4 & 1,2,2,2,2,4,4,4 & $X^{19/8}$ & $X^{1}$ & $X^{5/2}$\tabularnewline
\hline 
8T32 & 96 &  & Deg. 4 & 1,2,2,3,3,4,4,6 & $X^{23/8}$ & $X^{1/2}$ & $X^{5/2}$\tabularnewline
\hline 
8T33 & 96 & $(C_{2})^{2}\rtimes C_{6}$ & Deg. 2 & 1,2,2,3,4,4,4,6 & $X^{2}$ & $X^{1/2}$ & $X^{5/2}$\tabularnewline
\hline 
8T34 & 96 & $(E_{4})^{2}\rtimes D_{6}$ & Deg. 2 & 1,2,2,3,4,4,4,6 & $X^{2}$ & $X^{1/2}$ & $X^{5/2}$\tabularnewline
\hline 
8T35 & 128{*} &  & Deg. 4 & 1,2,2,2,3,4,4,8 & $X^{3}$ & $X^{1}$ & $X^{5/2}$\tabularnewline
\hline 
8T36 & 168 & $(C_{2})^{3}\rtimes F_{21}$ & none & 1,2,3,4,4,5,6,7 & $X^{15/7}$ & $X^{1/4}$ & $X^{5/2}$\tabularnewline
\hline 
8T37 & 168 & $PSL_{2}(\mathbb{F}_{7})$ & none & 1,2,3,4,4,4,6,7 & $X^{29/14}$ & $X^{1/4}$ & $X^{5/2}$\tabularnewline
\hline 
8T38 & 192 &  & Deg. 4 & 1,2,2,3,3,4,6,8 & $X^{27/8}$ & $X^{1}$ & $X^{5/2}$\tabularnewline
\hline 
8T39 & 192 &  & Deg. 4 & 1,2,2,3,3,4,4,6 & $X^{23/8}$ & $X^{1/2}$ & $X^{5/2}$\tabularnewline
\hline 
8T40 & 192 &  & Deg. 4 & 1,2,2,3,3,4,6,8 & $X^{27/8}$ & $X^{1/2}$ & $X^{5/2}$\tabularnewline
\hline 
8T41 & 192 & $(C_{2})^{3}\rtimes S_{4}$ & Deg. 2 & 1,2,2,3,4,4,4,6 & $X^{2}$ & $X^{1/2}$ & $X^{5/2}$\tabularnewline
\hline 
8T42 & 288 &  & Deg. 2 & 1,2,2,3,4,4,6,6 & $X^{13/6}$ & $X^{1/2}$ & $X^{5/2}$\tabularnewline
\hline 
8T43 & 336 & $PGL_{2}(\mathbb{F}_{7})$ & none & 1,2,3,4,4,6,7,8 & $X^{33/14}$ & $X^{1/3}$ & $X^{5/2}$\tabularnewline
\hline 
8T44 & 384 &  & Deg. 4 & 1,2,2,3,3,4,6,8 & $X^{27/8}$ & $X^{1}$ & $X^{5/2}$\tabularnewline
\hline 
8T45 & $2^{6}3^{2}$ &  & Deg. 2 & 1,2,2,3,4,4,6,8 & $X^{7/3}$ & $X^{1/2}$ & $X^{5/2}$\tabularnewline
\hline 
8T46 & $2^{6}3^{2}$ &  & Deg. 2 & 1,2,2,3,4,4,6,8 & $X^{7/3}$ & $X^{1/2}$ & $X^{5/2}$\tabularnewline
\hline 
8T47 & $2^{7}3^{2}$ &  & Deg. 2 & 1,2,2,3,4,4,6,8 & $X^{7/3}$ & $X^{1}$ & $X^{5/2}$\tabularnewline
\hline 
8T48 & $2^{6}3^{1}7^{1}$ & $AL(8)$ & none & 1,2,3,4,4,5,6,7 & $X^{15/7}$ & $X^{1/2}$ & $X^{5/2}$\tabularnewline
\hline 
8T49 & $8!/2$ & $A_{8}$ & none & 1,2,3,4,5,6,7,8 & $X^{17/7}$ & $X^{1/2}$ & $X^{5/2}$\tabularnewline
\hline
\end{longtable}

\begin{longtable}{|c|c|c|c|c|c|c|c|}
\hline 
\multicolumn{8}{|c|}{Some transitive subgroups of $S_{9}$}\tabularnewline
\hline 
\# & Order & Isom. to & Subfield? & Invariant Degrees & Result & Malle & Schmidt\tabularnewline
\hline 
\endfirsthead
\hline
\multicolumn{8}{|c|}{Some transitive subgroups of $S_{9}$ (continued)}\tabularnewline
\hline 
\# & Order & Isom. to & Subfield? & Invariant Degrees & Result & Malle & Schmidt\tabularnewline
\hline 
\endhead
9T3 & 18 & $D_{9}$ & Deg. 3 & 1,2,2,2,2,3,3,5,8 & $X^{13/6}$ & $X^{1/4}$ & $X^{11/4}$\tabularnewline
\hline 
9T4 & 18 & $S_{3}\times C_{3}$ & Deg. 3 & 1,2,2,2,3,3,3,3,6 & $X^{23/12}$ & $X^{1/3}$ & $X^{11/4}$\tabularnewline
\hline 
9T5 & 18{*} & $(C_{3})^{2}\rtimes C_{2}$ & Deg. 3 & 1,2,2,2,2,3,3,3,3 & $X^{19/12}$ & $X^{1/4}$ & $X^{11/4}$\tabularnewline
\hline 
9T8 & 36 & $S_{3}\times S_{3}$ & Deg. 3 & 1,2,2,2,3,3,3,4,6 & $X^{2}$ & $X^{1/3}$ & $X^{11/4}$\tabularnewline
\hline
\end{longtable}

\section*{Acknowledgements}

The author would like to thank Jordan Ellenberg and Akshay Venkatesh
for their work which inspired this paper, John Voight for computing assistance with MAGMA, David Dummit and Dinesh Thakur for their helpful editorial comments, and Jordan Ellenberg in particular for his comments and persistent encouragement.

\bibliography{cnfbib}

\begin{thebibliography}{10}

\bibitem{Baily}
A.~M. Baily, \emph{On the density of discriminants of quartic fields}, J. reine
  angew. Math \textbf{315} (1980) 190--210.

\bibitem{Belabas-Bhargava-Pomerance}
K.~Belabas, M.~Bhargava, and C.~Pomerance, \emph{Error estimates for the
  Davenport-Heilbronn theorems}, Duke Mathematical Journal \textbf{153} (2010),
  no.~1,  173--210.

\bibitem{Bhargava-3}
M.~Bhargava, \emph{The density of discriminants of quartic rings and fields},
  Annals of Mathematics  (2005) 1031--1063.

\bibitem{Bhargava-4}
---{}---{}---, \emph{The density of discriminants of quintic rings and fields},
  Annals of Mathematics  (2010) 1559--1591.

\bibitem{Bhargava-Shankar-Tsimerman}
M.~Bhargava, A.~Shankar, and J.~Tsimerman, \emph{On the Davenport--Heilbronn
  theorems and second order terms}, Inventiones Mathematicae \textbf{193}
  (2013), no.~2,  439--499.

\bibitem{Bhargava-Shankar-Wang}
M.~Bhargava, A.~Shankar, and X.~Wang, \emph{Geometry-of-numbers methods over
  global fields I: Prehomogeneous vector spaces}, arXiv preprint
  arXiv:1512.03035  (2015)\hskip -.1cm.

\bibitem{Bhargava-Wood}
M.~Bhargava and M.~Wood, \emph{The density of discriminants of $S_{3}$-sextic
  number fields}, Proceedings of the American Mathematical Society \textbf{136}
  (2008), no.~5,  1581--1587.

\bibitem{Cohen-1}
H.~Cohen, \emph{Constructing and Counting Number Fields}, Proceedings of the
  ICM \textbf{2} (2002) 129--138.

\bibitem{Cohen-2}
H.~Cohen, F.~D. y~Diaz, and M.~Olivier, \emph{A survey of discriminant
  counting}, in International Algorithmic Number Theory Symposium, 80--94,
  Springer (2002).

\bibitem{Cohen-Hilbert}
S.~D. Cohen, \emph{The distribution of Galois groups and Hilbert's
  irreducibility theorem}, Proceedings of the London Mathematical Society
  \textbf{3} (1981), no.~2,  227--250.

\bibitem{Conway-transitivegroups}
J.~H. Conway, A.~Hulpke, and J.~McKay, \emph{On transitive permutation groups},
  LMS Journal of Computation and Mathematics \textbf{1} (1998) 1--8.

\bibitem{Datskovsky-Wright}
B.~Datskovsky and D.~J. Wright, \emph{Density of discriminants of cubic
  extensions}, J. reine angew. Math \textbf{386} (1988) 116--138.

\bibitem{Davenport-Heilbronn}
H.~Davenport and H.~Heilbronn, \emph{On the density of discriminants of cubic
  fields. II}, Proceedings of the Royal Society of London. Series A,
  Mathematical and Physical Sciences  (1971) 405--420.

\bibitem{Derksen-Kemper}
H.~Derksen and G.~Kemper, Computational invariant theory, Springer (2015).

\bibitem{Dummit-rhodisc}
E.~P. Dummit, \emph{The $\rho$-discriminant and applications}. Unpublished
  preprint.

\bibitem{EV-functionfields}
J.~S. Ellenberg and A.~Venkatesh, \emph{Counting extensions of function fields
  with bounded discriminant and specified Galois group}, in Geometric methods
  in algebra and number theory, 151--168, Springer (2005).

\bibitem{EV-numberfields}
---{}---{}---, \emph{The number of extensions of a number field with fixed
  degree and bounded discriminant}, Annals of Mathematics  (2006) 723--741.

\bibitem{Jones-transitivegroups}
J.~Jones, Transitive Group Data (accessed October 2016).
  \href{http://hobbes.la.asu.edu/Groups/}{http://hobbes.la.asu.edu/Groups/}.

\bibitem{Kable-Yukie}
A.~C. Kable and A.~Yukie, \emph{On the number of quintic fields}, Inventiones
  Mathematicae \textbf{160} (2005), no.~2,  217--259.

\bibitem{Kluners}
J.~Kl{\"u}ners, \emph{A counter example to Malle's conjecture on the
  asymptotics of discriminants}, Comptes Rendus Mathematique \textbf{340}
  (2005), no.~6,  411--414.

\bibitem{Kluners-2}
---{}---{}---, \emph{The distribution of number fields with wreath products as
  Galois groups}, International Journal of Number Theory \textbf{8} (2012),
  no.~03,  845--858.

\bibitem{Kluners-Malle}
J.~Kl{\"u}ners and G.~Malle, \emph{Counting nilpotent Galois extensions}, J.
  reine angew. Math  (2004) 1--26.

\bibitem{Maki}
S.~M{\"a}ki, On the density of abelian number fields, Vol.~54, Suomalainen
  tiedeakatemia (1985).

\bibitem{Malle-1}
G.~Malle, \emph{On the distribution of Galois groups}, Journal of Number Theory
  \textbf{92} (2002), no.~2,  315--329.

\bibitem{Malle-2}
---{}---{}---, \emph{On the distribution of Galois groups, II}, Experimental
  Mathematics \textbf{13} (2004), no.~2,  129--135.

\bibitem{Schmidt}
W.~M. Schmidt, \emph{Number fields of given degree and bounded discriminant},
  Ast{\'e}risque \textbf{228} (1995), no.~4,  189--195.

\bibitem{Shankar-Tsimerman}
A.~Shankar and J.~Tsimerman, \emph{Counting $S_5$-fields with a power-saving
  error term}, in Forum of Mathematics, Sigma, Vol.~2, e13, Cambridge Univ
  Press (2014).

\bibitem{Siegel}
C.~L. Siegel, Lectures on the Geometry of Numbers, Springer Science \& Business
  Media (2013).

\bibitem{Taniguchi-Thorne}
T.~Taniguchi and F.~Thorne, \emph{Secondary terms in counting functions for
  cubic fields}, Duke Mathematical Journal \textbf{162} (2013), no.~13,
  2451--2508.

\bibitem{Turkelli}
S.~T{\"u}rkelli, \emph{Connected components of Hurwitz Schemes and Malle's
  conjecture}, Journal of Number Theory \textbf{155} (2015) 163--201.

\bibitem{Wright}
D.~J. Wright, \emph{Distribution of discriminants of abelian extensions},
  Proceedings of the London Mathematical Society \textbf{3} (1989), no.~1,
  17--50.

\bibitem{Yukie}
A.~Yukie, Shintani zeta functions, Vol. 183, Cambridge University Press (1993).

\end{thebibliography}
\bibliographystyle{mrl}

\end{document}